\documentclass{article}
\usepackage{todonotes}
\usepackage{arxiv}
\usepackage{amsmath,amssymb,amsthm}
\usepackage[utf8]{inputenc} 
\usepackage[T1]{fontenc}    
\usepackage{hyperref}       
\usepackage{url}            
\usepackage{booktabs}       
\usepackage{amsfonts}       
\usepackage{nicefrac}       
\usepackage{microtype}      
\usepackage{lipsum}		
\usepackage{graphicx}
\usepackage{natbib}
\usepackage{doi}
\usepackage{setspace}
\usepackage{color}

\numberwithin{equation}{section} 
\title{Well-posedness of Boussinesq Equation with Linear Restoring Force}

\date{} 					

\author{ Amin Esfahani\thanks{\href{mailto:amin.esfahani@nu.edu.kz}{amin.esfahani@nu.edu.kz}} \\
\textbf{Nazerke Zhardemova\thanks{\href{mailto:nazerke.zhardemova@nu.edu.kz}{nazerke.zhardemova@nu.edu.kz}}}  \\
	School of Computing and Artificial Intelligence (SCAI)\\ Nazarbayev University\\ Astana 010000, Kazakhstan
}

\renewcommand{\undertitle}{}

\newcommand{\R}{\mathbb{R}}

\newtheorem{thm}{Theorem}[section]
\newtheorem{lemma}[thm]{Lemma}

\begin{document}
\maketitle

\begin{abstract}
We establish a local well-posedness theory for the $n-$dimensional Boussinesq equation with a linear restoring force in a modified Sobolev spaces \(H_\omega^s\) and Bourgain-type spaces \(X_\omega^{s,b}\), which are adapted to its modified dispersion relation. The restoring force introduces a new phase function whose behavior  in the low-frequencies is different from ones of the classical Boussinesq equation, that in turn, affects the associated  energy functional. We prove bilinear and trilinear estimates for quadratic and cubic nonlinearities, respectively, which in turn imply local well-posedness under suitable
conditions on \(s\) and \(\omega\). To our knowledge, this is the first well-posedness result for this model.
\end{abstract}



\section{Introduction}

We study the following Cauchy problem for the Boussinesq equation with linear restoring force
 \begin{equation}\label{IVP}
        \begin{cases}
                u_{tt} = - \gamma u + \alpha \Delta u - \Delta^2 u + \Delta f(u),  \quad  (x,t)\in\mathbb{R}^n\times(0,\infty),\\
                u(x,0)=g(x) \\
                u_t(x,0)=h(x)
            \end{cases}
    \end{equation}
where $\gamma>  {\alpha^2}/{4}$, $\alpha\in\mathbb{R}$, and the nonlinear term is either $f(u)=\pm u^2$ or $f(u)=\pm u^3.$

Before discussing the restoring-force problem, let us review the previous studies in the context of the Boussinesq equation. The classical Boussinesq equation was initially derived by J.Boussinesq in \cite{Boussinesq1872} from Euler’s equations of motion for two-dimensional potential flow beneath a free surface, employing suitable approximations tailored for small amplitude long waves. Later it was referred as a ``bad'' Boussinesq equation due to its ill-posedness. Indeed, the bad Boussinesq equation is linearly unstable due to the exponentially growing Fourier components. In contrast, the terminology ``good'' refers to the positive sign of the fourth-order term, which gives a linearly stable dispersive model, given as
\[
    u_{tt}-u_{xx}+u_{xxxx}=(f(u))_{xx}.
\]
Moreover, it has also been derived in the examination of the dynamics of thin inviscid layers with free surfaces, as well as in the exploration of nonlinear strings, shape-memory alloys, wave propagation in elastic rods, and in the continuum limit of lattice dynamics or coupled electrical circuits.

Early well-posedness results for the ``good'' Boussinesq equation with a general nonlinearity were obtained by Bona and Sachs  \cite{BonaSachs1988}. They proved local well-posedness for initial data $(u_0,u_1)=(u_0,u_2')\in H^{s+2}\times H^{s+1}$ with $s>1/2$. Tsutsumi and Matahashi \cite{TsutsumiMatahashi1991} considered nonlinearities of the form
\(f(u)=|u|^{p-1}u\) and obtained the similar results for the data $ u_0\in H^1(\mathbb R)$ and $ u_1=\chi_{xx}$ with $\chi\in H^1(\mathbb R)$. Linares  \cite{Linares1993} improved these results by using Strichartz estimates for the linear problem and proved local well-posedness for $u_0\in L^2(\mathbb R),u_1=h_x, h\in H^{-1}(\mathbb R),$ for nonlinearities \(f(u)=|u|^{p-1}u\), \(1<p<5\). Fang and Grillakis subsequently established local and
global existence results for Boussinesq-type equations on the circle using
Fourier series and fixed-point arguments
\cite{FangGrillakis1996}. 

A significant low-regularity theory was developed using Bourgain-type Fourier
restriction spaces. Bourgain introduced a new space in his study of nonlinear
Schrödinger and KdV equations, and the method has subsequently been adapted to many other dispersive equations. For the good Boussinesq equation, Farah used these ideas associated with the phase $\varphi_0(\xi)=\sqrt{\xi^2+\xi^4},$
with norm
\[
    \|u\|_{X^{s,b}}
    =
    \left\|
    \langle |\tau|-\varphi_0(\xi)\rangle^b
    \langle \xi\rangle^s
    \widetilde u(\xi,\tau)
    \right\|_{L^2_{\xi,\tau}}.
\]
In this framework, Farah  \cite{Farah2009} proved local well-posedness for the quadratic good
Boussinesq equation with $ u_0=\phi\in H^s(\mathbb R), u_1=\psi_x, \psi\in H^{s-1}(\mathbb R)$ for $s>-1/4,$ and \(b>1/2\).
This was an important application of Bourgain-type
spaces to the Boussinesq equation and demonstrated that the dispersive
structure permits local solutions at negative Sobolev regularity. The threshold was subsequently improved. Kishimoto and Tsugawa developed
refined Fourier-restriction estimates and obtained local well-posedness for
the one-dimensional good Boussinesq equation in $H^s(\mathbb{R})\times H^{s-2}(\mathbb{R})$ for $s>- 1/2$; see \cite{KishimotoTsugawa2010}. Their approach exploits a reduction to a
quadratic nonlinear Schrödinger-type problem and refined bilinear estimates
in suitably modified \(X^{s,b}\) spaces. The endpoint was reached by Kishimoto \cite{Kishimoto2013}, who proved sharp local well-posedness at $s=-1/2$, and ill-posedness below this regularity, both on \(\mathbb{R}\) and \(\mathbb{T}\).

The global theory has also received substantial attention. For the cubic defocusing good Boussinesq equation, Farah and Linares proved global well-posedness in $H^s(\mathbb{R})$ for $2/3<s<1,$ using the \(I\)-method and an almost-conservation law \cite{FarahLinares2010}. Other works have established global existence,
scattering, and blow-up results for generalized Boussinesq equations under
various assumptions on the nonlinearity and the initial data  \cite{sachs,liu93-0,liu95,liu2000, Linares1993,linaressci,liu97} .

These results demonstrate that the Fourier restriction norm method is
particularly effective for the good Boussinesq equation. However, the phase in those works is $\varphi_0(\xi)=\sqrt{|\xi|^4+|\xi|^2},$ and hence its low-frequency behavior differs from that of
\eqref{IVP} with the corresponding linear phase 
\begin{equation}\label{phase-function}
        \varphi(\xi)
    =
    \sqrt{|\xi|^4+\alpha|\xi|^2+\gamma}.
\end{equation} The function spaces and low-frequency analysis therefore
cannot simply be transferred to the present equation without modification.

  It is worth noting that mass $M$ oscillates periodically with angular frequency $\omega = \sqrt{\gamma}$ and period $T = \frac{2\pi}{\sqrt{\gamma}}$, indeed,
  \[
  M(t) = M_0 \cos(\sqrt{\gamma} t) + \frac{H_0}{\sqrt{\gamma}} \sin(\sqrt{\gamma} t),\qquad
   M_0 = \int_\R g(x) dx,\,\;H_0 = \int_\R h(x) dx .
  \]
  When $\gamma=0<\alpha$, if initial mass and initial mass flux are both zero, then the total mass remains constant for all $t$, while it is not conserved in time when $\gamma>0$. Derived from the strictly positive dispersion relation, the phase velocity $v_p(\xi)$ dictates the speed of individual wave crests and exhibits a singularity as $|\xi|\to 0$, scaling as $\mathcal{O}(|\xi|^{-1})$. While this suggests infinitely fast phase propagation for long waves, the physical transport of energy is strictly constrained by the group velocity, $v_g(\xi)$. In the long-wave limit, the group velocity vanishes  $v_g \to 0$, demonstrating that low-frequency energy remains spatially localized and oscillates harmonically at frequency $\sqrt{\gamma}$ rather than radiating outward. Conversely, in the short-wave regime  $|\xi| \to \infty$, the biharmonic operator $-\Delta^2 $ heavily dominates the dynamics. Both velocities grow unboundedly, with the group velocity asymptoting to $v_g \sim 2|\xi|$. This pronounced high-frequency divergence indicates that short-wavelength wave packets are highly dispersive, rapidly shedding energy away from the localized low-frequency modes and driving the long-time asymptotic smoothing of the solution.

 Previous works on the Boussinesq equation with linear restoring force have investigated global weak solutions, potential-well dynamics, finite-time blow-up, and solitary waves. Kolkovska and   Vassilev  \cite{KolkovskaVassilev2019} considered
Boussinesq and double-dispersion equations of the form
\begin{equation}
U_{tt}-U_{xx}-\beta_1U_{ttxx}
+\beta_2U_{xxxx}
+mU
=
(f(U))_{xx},
\end{equation}
with \(m>0\). The term $mU$ is the linear restoring force and physically, it appears in models of elastic rods or beams lying on an elastic foundation, where the
foundation pulls the displacement back toward equilibrium with a force proportional
to the displacement. Kutev, Kolkovska, and Dimova in 
\cite{KutevKolkovskaDimova2014} established global existence and finite-time
blow-up results for solutions with subcritical energy and investigated the
sign-preserving properties of the associated Nehari functional. In subsequent work, Kutev, Kolkovska, and Dimova obtained finite-time
blow-up criteria for the linear restoring force equation with combined
power nonlinearities. In particular, their sufficient conditions apply to
initial data having arbitrarily large positive energy. The proofs use a
concavity argument together with suitable sign-preserving functionals
\cite{KutevKolkovskaDimova2016}.  The solitary-wave problem for the linear restoring force Boussinesq equation
has also been studied. Numerical investigations based on Petviashvili-type
iterations have produced solitary-wave solutions and examined their
properties \cite{KolkovskaVassilev2019}. In contrast, the low-regularity dispersive analysis through Bourgain spaces has not been developed for the linear restoring force model. A corresponding bilinear and trilinear estimates associated with the phase \eqref{phase-function} are not
obtained. Hence, our purpose  is to develop a low-regularity theory for the
linear restoring force problem in function spaces adapted to its particular
low-frequency and dispersive structure.

The phase \eqref{phase-function} exhibits two distinct frequency regimes. At high frequency $\varphi(\xi)\sim |\xi|^2,
$ while, since \(\gamma>0\), $\varphi(\xi)\sim\sqrt{\gamma}$ as $|\xi|\to0.$ On contrast, in this case, for the standard equation, $\varphi_0(\xi)\sim |\xi|$
and therefore they have a different low-frequency behavior. It is also reflected in the
natural energy associated with the equation. Formally, one obtains an energy conservation law of the form
\[
\begin{aligned}
E(u,u_t)
=\frac12\|(-\Delta)^{-\frac{1}{2}}u_t\|_{L^2}^2
+\frac{\gamma}{2}\|(-\Delta)^{-\frac{1}{2}}u\|_{L^2}^2 +\frac{\alpha}{2}\|u\|_{L^2}^2
+\frac12\|\nabla u\|_{L^2}^2
+\int_{\mathbb R^n}F(u)\,dx .
\end{aligned}
\]
with $F'(u)=f(u).$  The appearance of $(-\Delta)^{-1/2}$ in functional indicates that the natural energy contains a
negative-order component at low frequencies.

This motivates us to use the modified Sobolev space and Bourgain-type spaces adapted to the modified phase \eqref{phase-function}, given by
\[ 
    \begin{split}
         \|f\|_{H^s_\omega}
    &= 
    \left\|
    \langle \xi\rangle^{s+\omega}
    |\xi|^{-\omega}
    \widehat f(\xi)
    \right\|_{L^2_\xi}, 
  \\
        \|u\|_{X^{s,b}_\omega}
    &= 
    \left\|
    \langle \xi\rangle^{s+\omega}
    |\xi|^{-\omega}
    \langle |\tau|-\varphi(\xi)\rangle^b
    \widetilde u(\xi,\tau)
    \right\|_{L^2_{\xi,\tau}},
    \end{split}
\]
where $\widetilde u(\xi,\tau)=\mathcal{F}_{x,t}u(\xi,\tau).$
The parameter \(s\) controls the high-frequency regularity, whereas
\(\omega\) controls the behavior at the origin. Indeed,
\[
\langle\xi\rangle^{s+\omega}|\xi|^{-\omega}
\sim |\xi|^s
\qquad\text{for }|\xi|\gg1,
\]
\[
\langle\xi\rangle^{s+\omega}|\xi|^{-\omega}
\sim |\xi|^{-\omega}
\qquad\text{for }|\xi|\ll1.
\]
The choice \(\omega=1\) is particularly natural from the point of view of
the energy. In fact, for \(s=1\),
\[
\|u\|_{H_1^1}^2
\sim
\int_{\mathbb{R}^n}
\left(
|\xi|^2+1+|\xi|^{-2}
\right)
|\widehat u(\xi)|^2\,d\xi,
\]
which simultaneously captures the \(H^1\)-component and the
\(\dot H^{-1}\)-component appearing in the energy functional. 

Our main results are stated in the following theorems. 
\begin{thm}[Quadratic case]
Let \(\frac{1}{2}<b<1\), \(0<T<1\), and let \(f(u)=\pm u^2\). Assume that
\[
-\frac12<\omega<2,\qquad s>0,  \qquad \text{if } n=1,
\]
\[
1-\frac n2<\omega<2,\qquad s>\frac n2-1,  \qquad \text{if } n\ge2.
\]
Then, the Cauchy problem \eqref{IVP} is locally
well-posed in $H^s_\omega(\mathbb{R}^n)\times H^{s-2}_\omega(\mathbb{R}^n).$
\end{thm}

\begin{thm}[Cubic case]
Let \(\frac{1}{2}<b<1\), \(0<T<1\), and let \(f(u)=\pm u^3\). Assume that
\[
0<\omega<2,\qquad s>0, \qquad \text{if } n=1,
\]
\[
\max\left(-\frac n3,1-\frac n2\right)<\omega<2,
\qquad
s>\frac{n-1}{2}, \, \qquad \text{if } n\ge2.
\]
Then, the Cauchy problem \eqref{IVP} is locally
well-posed in $H^s_\omega(\mathbb{R}^n)\times H^{s-2}_\omega(\mathbb{R}^n).$
\end{thm}

The corresponding statements may be summarized as follows.

\begin{table}[h]
\centering
\begin{tabular}{c|c|c}
\hline
Nonlinearity & Dimension & Conditions \\
\hline
$f(u)=\pm u^2$
& $n=1$
& $-\frac12<\omega<2,\quad s>0$
\\[2mm]
$f(u)=\pm u^2$
& $n\ge2$
& $1-\frac n2<\omega<2,\quad s>\frac n2-1$
\\[2mm]
$f(u)=\pm u^3$
& $n=1$
& $0<\omega<2,\quad s>0$
\\[2mm]
$f(u)=\pm u^3$
& $n\ge2$
& $\max\{-\frac n3,1-\frac n2\}<\omega<2,
\quad s>\frac{n-1}{2}$
\\
\hline
\end{tabular}
\caption{Local well-posedness ranges.}
\label{intro:mainresults}
\end{table}

The proof proceeds in three steps. First, using Duhamel's principle, the solution to the equation \eqref{IVP} can be expressed in the following integral form
\begin{equation}\label{integral-form-sol}
u(t) = \partial_t \mathcal{W}_\phi(t) g + \mathcal{\omega} _\phi(t) h - \int_0^t \mathcal{\omega} _\phi(t - s) \Delta(f(u(s))) \, ds, 
\end{equation}
where the linear propagator defined as
\[
\mathcal{W}_\phi(t) = \frac{\sin(t \phi(D))}{\phi(D)} = \frac{e^{i t \phi(D)} - e^{-i t \phi(D)}}{2 i \phi(D)} = \frac{S_\phi(t) - S_\phi(-t)}{2 i \phi(D)}.
\]

Let \(\eta\in C_0^\infty(\mathbb R)\) be a smooth cutoff satisfying
\(\eta(t)=1\) for \(|t|\le1\), and define \(\eta_T(t)=\eta(t/T)\).
We define a map
\[
\Phi(u)(t)
=
\eta(t)\cos(t\phi(D))g
+
\eta(t)\frac{\sin(t\phi(D))}{\phi(D)}h
+
\eta_T(t)\int_0^t
\frac{\sin((t-\tau)\phi(D))}{\phi(D)}
\Delta(f(u(\tau)))\,d\tau .
\]
By the standard linear estimates in \(X_\omega^{s,b}\), we have
\[
\|\eta(t)\cos(t\phi(D))g\|_{X_\omega^{s,b}}
\lesssim
\|g\|_{H_\omega^s},
\]
and
\[
\left\|\eta(t)\frac{\sin(t\phi(D))}{\phi(D)}h\right\|_{X_\omega^{s,b}}
\lesssim
\|h\|_{H_\omega^{s-2}}.
\]
For the Duhamel term, using the standard \(X^{s,b}\) estimate and the equivalence
\(\phi(\xi)\sim \langle \xi\rangle^2\), we obtain
\[
\left\|
\eta_T(t)\int_0^t
\frac{\sin((t-\tau)\phi(D))}{\phi(D)}
\Delta(f(u(\tau)))\,d\tau
\right\|_{X_\omega^{s,b}}
\lesssim
\left\|
|D|^2\langle D\rangle^{-2}f(u)
\right\|_{X_\omega^{s,b-1}([0,T])}.
\]

To handle the nonlinearity in the last inequality, we need bilinear and trilinear estimates depending on the $f(u)$ in the adapted Bourgain spaces. The remainder of this paper is organized as follows. In Section \ref{section2}, we establish localized dispersive and Strichartz estimates for the modified linear group. In Section \ref{section3}, these bounds are used to prove the key bilinear and trilinear estimates in $X^{s,b}_\omega$. Finally,  a contraction mapping argument completes the proof of local well-posedness.

\section{Linear Estimates}\label{section2}

In this section, we will show that for high frequencies the phase \eqref{phase-function} satisfies the Schrödinger-type dispersive behavior. This motivates the localized
dispersive estimate, which then used to prove frequency-localized Strichartz estimates.

\begin{lemma} [Decay estimate for phase function]\label{phase-estimate} Consider the phase function of \eqref{IVP} given by $\phi(r)=\sqrt{r^4+\alpha r^2+\gamma}$ for $r=|\xi|.$ Then 
\[
\phi(r)\sim \langle r\rangle^2,
\qquad 
\langle r\rangle=(1+r^2)^{1/2}.
\]
In particular, this holds for \((\alpha,\gamma)=(1,1)\) and
\((\alpha,\gamma)=(-1,1)\).
Moreover,
\begin{enumerate}
\item If \(\alpha=1\) and \(\gamma=1\), then for all \(r\ge0\),
\[
\phi'(r)\sim r,
\qquad 
\phi''(r)\sim 1.
\]

\item If \(\alpha=-1\) and \(\gamma=1\), then for all \(r\ge1\),
\[
\phi'(r)\sim r,
\qquad 
\phi''(r)\sim 1,
\]
and for every integer \(j\ge3\),
\[
|\phi^{(j)}(r)|\le C_j r^{2-j}.
\]
\end{enumerate}
\end{lemma}

We recall several standard estimates that will be used throughout the paper. The first is the Bernstein  inequality, which provides a bound for derivatives of frequency-localized functions in terms of the corresponding Lebesgue norm.

\begin{lemma}[Bernstein inequality \cite{Grafakos2014}] \label{lem:bernstein-inequality}
For $1 \leq r \leq q \leq \infty$ and $k \geq 0$,
\[
\left\| D^k P_\lambda f \right\|_{L^q(\mathbb{R}^n)} \lesssim \lambda^{k + n\left(\frac{1}{r} - \frac{1}{q}\right)} \left\| P_\lambda f \right\|_{L^r(\mathbb{R}^n)}.
\]
\end{lemma}
We also recall the van der Corput lemma, which gives decay estimates for oscillatory integrals when the phase has a nonvanishing second derivative. 

\begin{lemma}[Van der Corput \cite{Stein1993}] \label{VanderCorput}
Assume $g \in C^1(a,b)$, $\psi \in C^2(a,b)$ and $|\psi''(r)| \geq A$ for all $r \in (a,b)$. Then
\begin{equation}
    \left| \int_a^b e^{it\psi(r)} g(r) \, dr \right|
\leq C(At)^{-1/2} \left[ |g(b)| + \int_a^b |g'(r)| \, dr \right],
\end{equation}
for some constant $C > 0$ that is independent of $a$, $b$ and $t$.
\end{lemma}

Finally, we use a higher-order version of the oscillatory integral estimate, which yields rapid decay in the oscillation parameter under suitable uniform bounds on the derivatives of the amplitude and the reciprocal of the derivative of the phase. These estimates are stated below for convenience.
\begin{lemma}\label{whenvandercorputdoesnotworklemma} \cite{Stein1993}
Suppose that $g \in C_0^\infty(a,b)$ and $\psi \in C^\infty(a,b)$ with $|\psi'(r)| > 0$ for all $r \in (a,b)$. If
\begin{equation}
    \max_{a \leq r \leq b} \left| \frac{d^j}{dr^j} g(r) \right| \leq A, \quad
\max_{a \leq r \leq b} \left| \frac{d^j}{dr^j} \left( \frac{1}{\psi'(r)} \right) \right| \leq B
\end{equation}
for all $j = 0, \dots, N \in \mathbb{N}$, then
\begin{equation}
    \left| \int_a^b e^{it\psi(r)} g(r) \, dr \right| \lesssim AB^N |t|^{-N}.
\end{equation}
\end{lemma}

\medskip
To capture the different behavior of low and high frequencies, we use the Littlewood–Paley decomposition, which allows us to decompose a function into dyadic frequency components. We introduce the corresponding notation as follows.
Fix an even function \(\chi\in C_0^\infty(\mathbb{R})\) such that
\(0\le \chi\le 1\), \(\chi(s)=1\) for \(|s|\le 1\), and
\(\operatorname{supp}\chi\subset[-2,2]\). Set
\[
\rho(s):=\chi(s)-\chi(2s), 
\qquad 
\rho_\lambda(s):=\rho\left(\frac{s}{\lambda}\right),
\qquad \lambda\in 2^{\mathbb{Z}}.
\]
Then
\[
\operatorname{supp}\rho_\lambda
\subset \left\{s\in\mathbb{R}:\frac{\lambda}{2}\le |s|\le 2\lambda\right\}.
\]
We define the Littlewood--Paley projection \(P_\lambda\) by
\[
\widehat{P_\lambda f}(\xi)
=
\rho_\lambda(|\xi|)\widehat f(\xi),
\]
and write \(f_\lambda:=P_\lambda f\). Thus, throughout the paper, we use the dyadic decomposition 
\[
f=\sum_{\lambda\in 2^{\mathbb{Z}}} f_\lambda .
\]
This frequency localization will be essential for applying the localized dispersive estimate given below.

\begin{lemma}[Localized dispersive estimate]
    Suppose $n \geq 1$ and $\lambda \gg 1$. Then for every $f \in \mathcal{S}(\mathbb{R}^n)$, 
\begin{equation}\label{lem:loc-dispersive-estimate}
    \| S_\phi(t) P_\lambda f \|_{L^\infty_x(\mathbb{R}^n)} \lesssim 
    |t|^{-\frac{n}{2}} \| f \|_{L^1_x(\mathbb{R}^n)}.
\end{equation}
\end{lemma}
\begin{proof}
Using   Fubini's theorem, we obtain
\begin{align*}
    [S_\phi(t) P_\lambda f](x) & = \int_{\R^n} I_\lambda (x-y, t) f(y) dy = \left ( I_\lambda (\cdot, t) * f \right ) (x),
\end{align*}
where 
    \begin{equation}
        I_\lambda(x,t)=\lambda^n \int_{\R^n} e^{i\lambda x \cdot \xi + i t\phi(\lambda\xi)}\rho(|\xi|) \; d\xi.
    \end{equation}
Therefore, by Young's inequality 
\begin{equation}
\| S_\phi(t) f_\lambda \|_{L_x^\infty(\mathbb{R}^n)} 
\leq \| I_\lambda(\cdot, t) \|_{L_x^\infty(\mathbb{R}^n)} 
\| f \|_{L_x^1(\mathbb{R}^n)}.
\end{equation}
Thus it remains to show that 
\begin{equation}\label{mid-goal}
\| I_\lambda(\cdot, t) \|_{L_x^\infty(\mathbb{R}^n)} 
\lesssim  |t| ^{-\frac{n}{2}}.
\end{equation}
Since $I(x, -t)=\overline{I(x,t)}$, it is enough to consider $t>0$. We prove only for $n\ge2$ and $\alpha=-1$, the cases $n=1$ and $\alpha=1$ are similar and simpler.

\noindent Using polar coordinates, we may write
\begin{equation}
I_\lambda(x,t) = \lambda^n \int_{\frac{1}{2}}^{2} e^{it\phi (\lambda r)} (\lambda r |x|)^{-\frac{n-2}{2}} J_{\frac{n-2}{2}}(\lambda r |x|) r^{n-1} \rho(r) \, \mathrm{d}r,
\end{equation}
where \( J_k(r) \) represents the Bessel function for $k>-\frac{1}{2}$.
We can represent the Bessel term using oscillatory components as follows
\begin{equation}\label{Bessel-decomposition}
s^{-\frac{n-2}{2}} J_{\frac{n-2}{2}}(s) = e^{is} h(s) + e^{-is} \overline{h}(s),
\end{equation}
where the function \( h(s) \) satisfies the decay property
\begin{equation}\label{bessel-decay}
\left| \frac{d^j}{ds^j} h(s) \right| \leq C_j \langle s \rangle^{-\frac{n-1}{2} - j} \quad \forall j \geq 0.
\end{equation}
For convenience, define
\[
\phi_\lambda(r) = \phi(\lambda r), \quad \widetilde{J}_a(r) = r^{-a} J_a(r), \quad \widetilde{\rho}(r) = r^{n-1} \rho(r),
\]
which allows us to simplify the \( I_\lambda(x,t) \) to the following form
\begin{equation}\label{simplified-integral}
I_\lambda(x,t) = \lambda^n \int_{1/2}^2 e^{it \phi_\lambda(r)} \widetilde{J}_{\frac{n-2}{2}}(\lambda r |x|) \widetilde{\rho}(r) \, dr.
\end{equation}
We now distinguish two cases.

\noindent \textbf{Case 1: \( \lambda |x| \lesssim 1 \).} 

\noindent By the Leibniz rule, for any $j \ge 0$,

\[
\partial_r^j \left[ \widetilde{J}_{\frac{n-2}{2}}(\lambda r|x|) \widetilde{\rho}(r) \right] = \sum_{k=0}^j \binom{j}{k} \partial_r^k \widetilde{J}_{\frac{n-2}{2}}(\lambda r|x|) \partial_r^{j-k} \widetilde{\rho}(r).
\]

\noindent First, consider $\partial_r^{j-k} \widetilde{\rho}(r)$, for $r \in \left(\frac{1}{2}, 2\right)$ we have $\rho(r) \in C_0^\infty\left( \frac{1}{2}, 2 \right).$ So each term of the following sum is bounded
\[
\partial_r^m \widetilde{\rho}(r) = \sum_{k=0}^m \binom{m}{k} \partial_r^k r^{n-1} \partial_r^{m-k} \rho(r).
\]
Hence,
\[
|\partial_r^m \widetilde{\rho}(r)| \mathop{\lesssim}\limits_{m} \: 1, \quad \forall r \in \left(\frac{1}{2}, 2\right).
\]
Next, consider $\partial_r^k \widetilde{J}_{\frac{n-2}{2}}(\lambda r|x|)$. Since $r > \frac{1}{2}$ and $\lambda |x| \leq 1$, using \eqref{Bessel-decomposition} and \eqref{bessel-decay}, we see that the Bessel function and its derivatives are uniformly bounded.
Therefore
\begin{equation}\label{bessel-derivative-bound}
\left| \partial_r^j \left[ \widetilde{J}_{\frac{n-2}{2}}(\lambda r |x|) \widetilde{\rho}(r) \right] \right| \mathop{\lesssim}\limits_{j} \: 1  \quad \forall j \geq 0.
\end{equation}
Therefore, Lemma \ref{phase-estimate} implies 
\begin{equation}\label{phi-derivative-bound}
\max_{1/2 \leq r \leq 1} \left| \frac{d^j}{dr^j} \left( \frac{1}{\phi'_\lambda(r)} \right) \right| \mathop{\lesssim}\limits_{j} \lambda^{-2} \quad \forall j \geq 0.
\end{equation}
Then, using Lemma \ref{whenvandercorputdoesnotworklemma} and combining the bounds from \eqref{bessel-derivative-bound} and \eqref{phi-derivative-bound}, along with the expression \eqref{simplified-integral} and the choice \( N = \frac{n}{2} \), we find:
\begin{equation}\label{case1}
|I_\lambda(x,t)| \lesssim \lambda^n \cdot \lambda^{-\frac{2n}{2}} t^{-\frac{n}{2}} \lesssim |t|^{-\frac{n}{2}}.
\end{equation}

\noindent \textbf{Case 2: \( \lambda |x| \gg 1 \).}

\noindent Since the Bessel function term admits the representation \eqref{Bessel-decomposition}, substituting into $I_\lambda$, we split the integral as
\[
I_\lambda(x,t) = I_\lambda^+(x,t) + I_\lambda^-(x,t),
\]
where
\[
I_\lambda^+(x,t) = \lambda^n \int_{1/2}^2 e^{it\theta_\lambda^+(r)} H_\lambda (|x|,r) \, dr,
\qquad
I_\lambda^-(x,t) = \lambda^n \int_{1/2}^2 e^{-it\theta_\lambda^-(r)} \bar{H}_\lambda (|x|,r) \, dr,
\]
with
\[
\theta_\lambda^\pm(r) := \frac{\lambda r |x|}{t}\pm \phi(\lambda r), \qquad H_\lambda(|x|, r) := h(\lambda r |x|)\widetilde{\rho}(r).
\]
Using \eqref{bessel-decay} and $\lambda |x| \gg 1$,
\[
\left| \partial_r^j h(\lambda r |x|) \right| \lesssim (\lambda |x|)^j \cdot (\lambda r |x|)^{-\frac{n-1}{2} - j}
\lesssim (\lambda |x|)^{-\frac{n-1}{2}}.
\]
Hence,
\begin{equation}\label{bound-H}
    \max_{\frac{1}{2} \leq r \leq 2} \left| \partial_r^j H_\lambda(|x|, r) \right|
\lesssim (\lambda |x|)^{-\frac{n-1}{2}}.
\end{equation}
Using the chain rule
\[
\partial_r \theta_\lambda^\pm(r) = \lambda \left (\frac{|x|}{t}\pm \phi'(\lambda r) \right ), \qquad
\partial_r^2 \theta_\lambda^\pm(r) = \pm \lambda^2 \phi''(\lambda r).
\]
From Lemma \ref{phase-estimate}, we know that \( \phi'(\lambda r) \sim \lambda \), and \( \phi''(\lambda r) \sim 1 \), uniformly for \( r \in (\frac{1}{2}, 2) \) and \( \lambda \gg 1 \). Therefore
\begin{equation}\label{bound-theta}
    |\partial_r \theta_\lambda^+(r)| \gtrsim \lambda^2, \qquad
|\partial_r^2 \theta_\lambda^\pm(r)| \sim \lambda^2.
\end{equation}

\noindent \textit{Estimate of \( I_\lambda^+(x,t) \)}. It is straightforward that
\begin{equation}\label{bound-derivative-theta}
    \max_{\frac{1}{2} \leq r \leq 2} \left| \partial_r^j \left( \frac{1}{\partial_r \theta_\lambda^+(r)} \right) \right| \mathop{\lesssim}\limits_{j} \: \lambda^{-2} \quad \forall j \geq 0. 
\end{equation}
Using Lemma \ref{whenvandercorputdoesnotworklemma} with \( N = \frac{n}{2} \), estimates \eqref{bound-H} and \eqref{bound-derivative-theta}, we get
\begin{equation}\label{case2-plus}
|I_\lambda^+(x,t)| 
\lesssim |t|^{-\frac{n}{2}}.
\end{equation}

\noindent \textit{Estimate of \( I_\lambda^-(x,t) \)}. Now, consider 
\[
\partial_r \theta_\lambda^-(r) = \lambda \frac{|x|}{t}- \lambda \phi'(\lambda r),
\]
since \( \phi'(\lambda r) \sim \lambda \), a stationary point can occur at $|x| \sim \lambda t.$
In this case, we apply Lemma \ref{VanderCorput}, using the second derivative from \eqref{bound-theta} and \eqref{bound-H} to obtain
\begin{equation}\label{case2-minus}
\begin{split}
|I_\lambda^-(x,t)|
&\lesssim \lambda^n (\lambda^2 t)^{-\frac{1}{2}} 
\left[
|H_\lambda^-(x,2)| + \int_{1/2}^2 \left| \partial_r H_\lambda^-(x,r) \right| \, dr
\right] \lesssim |t|^{-\frac{n}{2}}, 
\end{split}
\end{equation}
since \( H_\lambda^-(x,2) = 0 \).

\noindent It remains to consider the nonstationary cases, when $|x| \ll  \lambda t  \text{ or } |x| \gg \lambda t.$ In this case,
\[
|\partial_r \theta_\lambda^-(r)| = \left| \lambda \frac{|x|}{t} - \lambda \phi'(\lambda r) \right| \gtrsim \lambda^2.
\]
Therefore, exactly as in the estimate of \( I_\lambda^+(x,t) \), Lemma \ref{whenvandercorputdoesnotworklemma} yields
\[
|I_\lambda^-(x,t)| \lesssim |t|^{-\frac{n}{2}}. 
\]
\end{proof}

Now, we establish frequency-localized Strichartz estimates for the linear group $S_\phi(t)= e^{it\phi(D)}$, which will serve as a fundamental tool in the bilinear and trilinear estimates developed later.

\begin{lemma}[Localized Strichartz estimates]\label{lem:localized-strichartz}
Let $\lambda\gg1$, and assume that the pair $(q,r)$ satisfies
\begin{equation}\label{eq:2.26}
q>2,\quad r\ge2,
\quad\text{and}\quad
\frac2q+\frac n r=\frac n2.
\end{equation}
Then
\begin{equation}\label{strichartz-estimate}
\|S_{\phi}(t)f_\lambda\|_{L^q_tL^r_x(\mathbb{R}^{n+1})}
\;\lesssim\;\|f_\lambda\|_{L^2_x(\mathbb{R}^n)},
\end{equation}
for all $f\in\mathcal S(\mathbb{R}^n)$.  Moreover, if $b>\tfrac12$, then
\begin{equation}\label{strichartz-estimate-bourgain}
\|u_\lambda\|_{L^q_tL^r_x(\mathbb{R}^{n+1})}
\;\lesssim\;
\|u_\lambda\|_{X^{0,b}_\omega}.
\end{equation}
\end{lemma}

\begin{proof}
In the case when $(q,r) = (\infty,2)$, Plancherel theorem yields the desired result.
Hence, from now on we assume $2 < q < \infty$. Let the linear solution operator defined by
\[
T f(t, x) = S_{\phi}(t) P_\lambda f(x).
\]
We will prove that
\[
T : L_x^2(\mathbb{R}^n) \to L_t^q L_x^r(\mathbb{R}^{n+1})
\]
is bounded using the $TT^*$ argument.
So, we prove the equivalent estimate
\begin{equation}\label{TT*-estimate}
    \|T T^* F\|_{L^q_t L^r_x(\R^{n+1})} \lesssim  \|F\|_{L^{q'}_t L^{r'}_x (\R^{n+1})}.
\end{equation}
The adjoint $T^*$ is defined by
\[
T^* F(x) = \int_{\mathbb{R}} \int_{\mathbb{R}^n} e^{-i(x \cdot \xi + s \phi(\xi))} \rho_\lambda(|\xi|) \hat{F}(s, \xi) \, d\xi \, ds.
\]
Applying \( T \) to \( T^* F \) and expressing it as a space-time convolution  
\begin{equation}\label{TT*-argument}
\begin{split}
(T T^* F)(t,x)
&= \int_{\mathbb{R}} \int_{\mathbb{R}^n} e^{i(x \cdot \xi + (t - s) \phi(\xi))} \rho_\lambda^2(|\xi|) \hat{F}(s, \xi) \, d\xi \, ds \\
&=\int_{\mathbb{R}} \left( K_\lambda(t - s, \cdot) * F(s, \cdot) \right)(x) \, ds,
\end{split}
\end{equation}
where the convolution kernel is
\[
K_\lambda(t, x) = \int_{\mathbb{R}^n} e^{i(x \cdot \xi + t \phi(\xi))} \rho_\lambda^2(|\xi|) \, d\xi.
\]
Note that 
\[
K_\lambda(x, t) * h(x)= S_{\phi}(t) P_\lambda h_\lambda (x).
\]
Then, \eqref{lem:loc-dispersive-estimate} implies that 
\begin{equation}\label{dis-estimate}
\|K_\lambda(t) * h\|_{L^\infty_x (\R^n)} \lesssim  (|t|)^{-\frac{n}{2}} \| h \|_{L^1_x(\mathbb{R}^n)}.
\end{equation}
Using the Plancherel identity, we also have
\begin{equation}\label{plancherel}
    \|K_\lambda(t) * h\|_{L^2_x} \lesssim \|h\|_{L^2_x (\R ^n)}.
\end{equation}
Applying the  Riesz-Thorin interpolation theorem to \eqref{dis-estimate} and \eqref{plancherel}, let $0 \leq \theta \leq 1$,
\[
\|K_\lambda(t,\cdot) * h\|_{L_x^p}
\lesssim |t|^{-\frac{n}{2}\theta}\|h\|_{L_x^{p'}},
\]
where
\[
\frac{1}{p}=\frac{1-\theta}{2},
\]
so we have $\theta = 1-\frac{2}{r}.$
Therefore
\begin{equation}\label{interpolation-mid-step}
    \|K_\lambda(t,\cdot) * h\|_{L_x^r}
\lesssim |t|^{-\frac{n}{2}\left(1-\frac{2}{r}\right)}\|h\|_{L_x^{r'}}.
\end{equation}
Using the admissibility condition, \eqref{interpolation-mid-step} becomes
\begin{equation}\label{interpolation}
    \|K_\lambda(t,\cdot) * h\|_{L_x^r}
\lesssim |t|^{-\frac{2}{q}}\|h\|_{L_x^{r'}}.
\end{equation}
Now we use the one-dimensional Hardy--Littlewood--Sobolev (HLS) inequality in the time variable. Since $q>2$, we have $0<\frac{2}{q}<1.$ Moreover, if we set $\gamma=\frac{2}{q}$, then
\[
\frac{1}{q'}
=1-\frac{1}{q}
=\frac{1}{q}+1-\gamma.
\]
Thus rhe HLS inequality is satisfied with $a=q, b=q',\gamma=\frac{2}{q}.$
Using the Minkowski’s inequality, \eqref{interpolation} and HLS, we get

\[
\begin{aligned}
\|TT^* F\|_{L^q_t L^r_x(\mathbb{R}^{n+1})}
&= \left\| \int_{\mathbb{R}} K_\lambda(t - s, \cdot) * F(s, \cdot) \, ds \right\|_{L^q_t L^r_x} \\
&\leq \left\| \int_{\mathbb{R}} \| K_\lambda(t - s, \cdot) * F(s, \cdot) \|_{L^r_x} \, ds \right\|_{L^q_t} \\
&\lesssim  \left\| \int_{\mathbb{R}} |t - s|^{-\frac{2}{q}} \|F(s, \cdot)\|_{L^{r'}_x(\mathbb{R}^n)} \, ds \right\|_{L^q_t(\mathbb{R})} \\
&=  \|F\|_{L^{q'}_t L^{r'}_x(\mathbb{R}^{n+1})}.
\end{aligned}
\]
Then \eqref{strichartz-estimate-bourgain} follows from \eqref{strichartz-estimate} by the same argument as in  \cite[Lemma 2.9]{tao2001multilinear}.
\end{proof}

\section{Multi-linear estimates}\label{section3}

In this section, we establish the bilinear and trilinear estimates needed to control the quadratic and cubic nonlinearities, respectively. The main tool is the frequency-localized Strichartz estimate proved in the previous section, combined with dyadic decompositions. These estimates are designed
to control the nonlinear term in the Duhamel formulation.

\begin{lemma}\label{lem:bilinear-estimate}
Let $n \geq 1$, $\frac{1}{2}< b < 1$, and $0 < T < 1$. Assume that
\[
-\frac{1}{2} < \omega < 2,\qquad s > 0 \quad \text{if } n=1,
\]
and
\[
1-\frac{n}{2} < \omega < 2,\qquad s > \frac{n}{2}-1 \quad \text{if } n \geq 2.
\]
Then
\begin{equation}\label{bilinear-estimate}
    \left\| |D|^2\langle D \rangle^{-2}(u_1 u_2) \right\|_{X^{s,b-1}_\omega (0,T)} \lesssim T^{1-b} \left\| u_1 \right\|_{X^{s,b}_\omega(0,T) } \left\| u_2 \right\|_{X^{s,b}_\omega (0,T)}.
\end{equation}
\end{lemma}

The next lemma is the frequency-localized bilinear estimate. It will be used to prove Lemma \ref{lem:bilinear-estimate}.

\begin{lemma}\label{lem:loc-bilinear-estimate}
Let $n \geq 1$, $\frac{1}{2}< b < 1$, $0 < T < 1$ and $\lambda_1, \lambda_2, \lambda_3$ be dyadic numbers. Then
\begin{equation}\label{eq:biliner-local}
    \left\| |D|^2\langle D \rangle^{-2} P_{\lambda_3} \left( P_{\lambda_1} u_1 P_{\lambda_2} u_2 \right) \right\|_{L^2_T L^2_x} 
\lesssim \lambda_3^2 \langle \lambda_3 \rangle^{-2} B(\lambda) 
\left\| P_{\lambda_1} u_1 \right\|_{X^{0,b}_0} 
\left\| P_{\lambda_2} u_2 \right\|_{X^{0,b}_0},
\end{equation}
where
\begin{equation}\label{eq:bilinear-ball}
B(\lambda) \sim \left[ \min(\lambda_1, \lambda_2) \right]^{\frac{n}{2}}.
\end{equation}
Moreover, if $\max(\lambda_1, \lambda_2) \gg 1$, then improved $B$ is
\begin{equation}\label{eq:bilinear_ball_improved}
B(\lambda) \sim 
\begin{cases}
1 & \text{if } n = 1, \\
\left[ \min(\lambda_1, \lambda_2) \right]^{\frac{n}{2} - 1 + \delta}
 & \text{if } n \geq 2,
\end{cases}
\end{equation}
for sufficiently small $\delta > 0$.
\end{lemma}

\begin{proof}
Set $F:=P_{\lambda_1}u_1\cdot P_{\lambda_2}u_2$. By the Bernstein inequality \ref{lem:bernstein-inequality}, 
\[
\bigl\|\,|D|^{2}\langle D\rangle^{-2} P_{\lambda_3}F\bigr\|_{L_T^2L_x^2}
\;\lesssim\; \lambda_3^{2}\langle \lambda_3\rangle^{-2}\, \|P_{\lambda_3}F\|_{L_T^2L_x^2}
\;\le\; \lambda_3^{2}\langle \lambda_3\rangle^{-2}\, \|F\|_{L_T^2L_x^2}.
\]
Thus it suffices to prove
\begin{equation}\label{bilinear-core}
  \|P_{\lambda_1}u_1 P_{\lambda_2}u_2\|_{L_T^2L_x^2}
  \ \lesssim\ B(\lambda)\,\|P_{\lambda_1}u_1\|_{X^{0,b}_0}\;\|P_{\lambda_2}u_2\|_{X^{0,b}_0}.
\end{equation}
Assuming $\lambda_1 \leq \lambda_2$, using the Hölder inequality
\[
  \|P_{\lambda_1}u_1 P_{\lambda_2}u_2\|_{L_T^2L_x^2}
  \ \le\ \|P_{\lambda_1} u_1\|_{L_T^2L_x^\infty}\;
        \|P_{\lambda_2}u_2\|_{L_T^\infty L_x^2}.
\]
Apply the Bernstein inequality (\ref{lem:bernstein-inequality}) to the first factor and Sobolev embedding to both,
then use $T^{\frac{1}{2}} < 1$
\[
\|P_{\lambda_1}u_1\|_{L_T^2L_x^\infty}
  \ \lesssim\ \lambda_1^{\frac{n}{2}}\,\|P_{\lambda_1}u_1\|_{L_T^2L_x^2}
  \ \lesssim\ \lambda_1^{\frac{n}{2}}\,\|P_{\lambda_1}u_1\|_{L_T^\infty L_x^2}
  \ \lesssim\ \lambda_1^{\frac{n}{2}}\,\|P_{\lambda_1}u_1\|_{X^{0,b}_0}.
\]
Hence
\[
  \|P_{\lambda_1}u_1 P_{\lambda_2}u_2\|_{L_T^2L_x^2}
  \ \lesssim\ \lambda_1^{\frac{n}{2}}\,
  \|P_{\lambda_1}u_1\|_{X^{0,b}_0}\;\|P_{\lambda_2}u_2\|_{X^{0,b}_0},
\]
which proves \eqref{eq:biliner-local} together with \eqref{eq:bilinear-ball} .

We now prove the improved bound under the additional assumption $\lambda_{2} \gg 1$. 

\noindent \textbf{Case $n=1$.} In one dimension the admissibility condition is $\tfrac{2}{q}+\tfrac{1}{r}=\tfrac12$. Choose $(q,r)=(4,\infty)$.
Then by Lemma \ref{lem:localized-strichartz} 
\[
\|P_{\lambda_{2}}u_{2}\|_{L_T^{4}L_x^{\infty}}
  \ \lesssim\ \|P_{\lambda_{2}}u_{2}\|_{X^{0,b}_0}.
\]
Using this and the Hölder inequality,
\begin{equation*}
\begin{split}
    \|P_{\lambda_1}u_1 P_{\lambda_2}u_2\|_{L_T^2L_x^2}
   & \ \le\ 
   \|P_{\lambda_{1}}u_{1}\|_{L_T^{4}L_x^{2}} \;
  \|P_{\lambda_{2}}u_{2}\|_{L_T^{4}L_x^{\infty}} \\
  & \ \lesssim\  T^{\frac{1}{4}} \|P_{\lambda_{1}}u_{1}\|_{L_T^\infty L_x^{2}} \;
  \|P_{\lambda_{2}}u_{2}\|_{L_T^{4}L_x^{\infty}} \\
  & \ \lesssim\  T^{\frac{1}{4}}   \|P_{\lambda_{1}}u_{1}\|_{X^{0,b}_0}  \|P_{\lambda_{2}}u_{2}\|_{X^{0,b}_0}.
\end{split}
\end{equation*}

\noindent \textbf{Case $n\ge2$.} Fix a small $\delta\in(0,\frac14]$ and define
\[
  r:=\frac{2n}{\,n-2+2\delta\,},\qquad
  q:=\frac{2}{1-\delta}.
\]
Then $(q,r)$ is Strichartz admissible and Lemma \ref{lem:localized-strichartz} provides
\[
  \|P_{\lambda_{2}}u_{2}\|_{L_T^{q}L_x^{r}}
  \ \lesssim\ 
  \|P_{\lambda_{2}}u_{2}\|_{X^{0,b}_0}.
\]
The Hölder inequality and the previous inequality yield
\begin{equation}\label{loc-bilin-mid-step}
    \begin{split}
\|P_{\lambda_1}u_1 P_{\lambda_2}u_2\|_{L_T^2L_x^2}
  \ \le\ T^{\frac{1}{2}-\frac{1}{q}} \ \|P_{\lambda_{1}}u_{1}\|_{L_T^\infty L_x^{\frac{qn}{2}}} \|P_{\lambda_{2}}u_{2}\|_{X^{0,b}_0}.
\end{split}
\end{equation}
Using the Bernstein inequality \ref{lem:bernstein-inequality} with Sobolev embedding,
\begin{equation}\label{loc-bilin-mid-step2}
      \|P_{\lambda_{1}}u_{1}\|_{L_T^\infty L_x^{\frac{qn}{2}}}
  \ \lesssim\ \lambda_{1}^{\,n\left(\frac12-\frac2{qn}\right)} \|P_{\lambda_{1}}u_{1}\|_{L_T^\infty L_x^2}
  \ =\ \lambda_{1}^{\,\frac n2-1+\delta}\,\|P_{\lambda_{1}}u_{1}\|_{X^{0,b}_0}.
\end{equation}
Substituting \eqref{loc-bilin-mid-step2} into \eqref{loc-bilin-mid-step}, and using $T<1$, for $n\ge2$ we get
\[
  \|P_{\lambda_1}u_1 P_{\lambda_2}u_2\|_{L_T^2L_x^2}
  \ \lesssim\ T^{\frac{\delta}{2}} \lambda_{1}^{\frac n2-1+\delta}\,
  \|P_{\lambda_1}u_1\|_{X^{0,b}_0}\ \|P_{\lambda_2}u_2\|_{X^{0,b}_0}. 
  \]
\end{proof}

To prove Lemma \ref{lem:bilinear-estimate}, we use the standard restriction-in-time estimate \cite{tao2006nonlinear}
\begin{equation}\label{tao-eq}
    \|v\|_{X_\omega^{s,b-1}([0,T])}
\lesssim
T^{1-b}\|v\|_{L^2([0,T];H_\omega^s)},
\qquad \frac{1}{2}<b<1,\ 0<T<1.
\end{equation}
Applying this to $v=|D|^2\langle D\rangle^{-2}(u_1u_2)$, it is enough to prove
\begin{equation}\label{step-1}
\bigl\||D|^2\langle D\rangle^{-2}(u_1u_2)\bigr\|_{L^2_T H^s_{\omega} }
\;\lesssim\;\|u_1\|_{X^{s,b}_{\omega} }\,\|u_2\|_{X^{s,b}_{\omega} ([0,T])}. 
\end{equation}

\medskip

Now, we are in a position to prove Lemma Lemma \ref{lem:bilinear-estimate}.
\begin{proof}
[Proof of Lemma \ref{lem:bilinear-estimate}] 
 
Using the definition of $H^s_\omega$ space and duality principle
\[
\Big\|\,|D|^{2-\omega}\langle D\rangle^{s+\omega-2}(u_1u_2)\,\Big\|_{L^2_{T}L^2_x}
=
\sup_{u_3 \in L^2_{t,x}}\frac{
\left|
\int_0^T\!\!\int_{\mathbb{R}^n}
|D|^{2-\omega}\langle D\rangle^{s+\omega-2}(u_1u_2)\,u_3
\,dx\,dt \right|}{\|u_3\|_{L^2_T L^2_x}}.
\]
Then, by definition of $X^{s,b}_\omega$ space we need to prove 
\begin{equation}\label{step-3}
\begin{split}
    &\sup \left|
\int_0^T\!\!\int_{\mathbb{R}^n}
 |D|^{2-\omega}\langle D\rangle^{s+\omega-2}
\big(|D|^{\omega}\langle D\rangle^{-s-\omega}u_1\, \cdot |D|^{\omega}\langle D\rangle^{-s-\omega}u_2\big)\,
u_3
\,dx\,dt \right|\\&\qquad
\lesssim
\|u_1\|_{X^{0,b}_0}\,
\|u_2\|_{X^{0,b}_0}
\|u_3\|_{L^2_T L^2_x}.
\end{split}
\end{equation}
Applying the Littlewood-Paley decomposition $u_j=\sum_{\lambda_j}P_{\lambda_j}u_j$ ($j=1,2,3$),  and using the Cauchy-Schwarz inequality, Lemma \ref{lem:loc-bilinear-estimate}, and the Bernstein inequality it suffices to show the last inequality
\begin{equation}
\begin{split}
& \sum_{\lambda_1, \lambda_2, \lambda_3}
\left|
\int_{0}^{T}\!\!\int_{\mathbb{R}^n}
\,|D|^{2-\omega}\,\langle D\rangle^{s+\omega-2}\,
P_{\lambda_3}\Big(|D|^{\omega}\langle D\rangle^{-s-\omega}P_{\lambda_1}u_1|D|^{\omega}\langle D\rangle^{-s-\omega}P_{\lambda_2}u_2\Big)P_{\lambda_3}u_3
\,dx\,dt
\right| \\
\ & \lesssim\ 
\sum_{\lambda_{1},\lambda_{2},\lambda_{3}}
\big\|\,|D|^{2-\omega}\langle D\rangle^{s+\omega-2}\,
P_{\lambda_3}\big(|D|^{\omega}\langle D\rangle^{-s-\omega}P_{\lambda_1}u_1 |D|^{\omega}\langle D\rangle^{-s-\omega}P_{\lambda_2}u_2\big)\big\|_{L^2_{T,x}} \|P_{\lambda_3}u_3\|_{L^2_{T,x}} \\
& \lesssim \underbrace{ \sum_{\lambda_{1},\lambda_{2},\lambda_{3}} G(\lambda)\,
a_{\lambda_1}\,a_{\lambda_2}\,b_{\lambda_3}}_{:=R}  \lesssim\ \|(a_{\lambda_1})\|_{l_{\lambda_1}^2}\,
\|(a_{\lambda_2})\|_{l_{\lambda_2}^2}\,\|(b_{\lambda_3})\|_{l_{\lambda_3}^2},
\end{split}
\end{equation}
where
\[
a_{\lambda_j}:=\|P_{\lambda_j}u_j\|_{X_\omega^{s,b}},\qquad
b_{\lambda_3}:=\|P_{\lambda_3}u_3\|_{L_T^2L_x^2},
\]
and
\[
G(\lambda)
:=B(\lambda)\lambda_3^{2-\omega}\langle\lambda_3\rangle^{s+\omega-2}
\prod_{j=1}^2 \lambda_j^\omega\langle\lambda_j\rangle^{-s-\omega} .
\]

By symmetry we may assume $\lambda_1\leq \lambda_2$. It implies that only the following interactions need to be considered
\begin{enumerate}
\item $\lambda_2\lesssim 1$;
\item $\lambda_2\gg 1$ and $\lambda_3\ll \lambda_2$, in which case $\lambda_1\sim \lambda_2$;
\item $\lambda_2\gg 1$ and $\lambda_3\sim \lambda_2$, in which case $\lambda_1 \le \lambda_2$.
\end{enumerate}

\textbf{Case 1: $\lambda_2\lesssim 1$}. Use the rough bound \eqref{eq:bilinear-ball} to obtain
\[
G(\lambda)\ \lesssim\ \lambda_1^{\frac{n}{2}+\omega}\lambda_2^\omega\lambda_3^{2-\omega}. \]
Hence, by the Cauchy--Schwarz inequality
\[
R \lesssim
\|(a_{\lambda_1})\|_{\ell^2_{\lambda_1}}
\|(a_{\lambda_2})\|_{\ell^2_{\lambda_2}}
\|(b_{\lambda_3})\|_{\ell^2_{\lambda_3}} ,
\]
provided
\[
-\frac{n}{2}<\omega<2.
\]

\textbf{Case 2: $\lambda_2\gg 1$ and $\lambda_3\ll \lambda_2$}. In this case $\lambda_1\sim \lambda_2$, which gives high$\times$high$\to$low interaction.

\textbf{$n=1$.} Then $B(\lambda)\sim 1$, and therefore
\[
G(\lambda)\sim \lambda_3^{2-\omega}\langle \lambda_3\rangle^{s+\omega-2}\lambda_2^{-2s}.
\]

\begin{itemize}
\item  If $\lambda_3 \lesssim 1$, then by Cauchy--Schwartz
\[
R \lesssim
\sum_{\lambda_3 \lesssim 1 \ll \lambda_1 \sim \lambda_2}
\lambda_3^{2-\omega} \,
\lambda_2^{-2s}
\, a_{\lambda_1} a_{\lambda_2} b_{\lambda_3}
\;\le\;
\|(a_{\lambda_1})\|_{\ell^2_{\lambda_1}}
\|(a_{\lambda_2})\|_{\ell^2_{\lambda_2}}
\|(b_{\lambda_3})\|_{\ell^2_{\lambda_3}}
\]
provided
\[
\omega<2,\qquad s\ge0.
\]
\item If $\lambda_3 \gg 1$, 

\[
R \lesssim 
\sum_{1 \ll \lambda_3 \ll \lambda_1 \sim \lambda_2}
\lambda_3^{s}\,
\lambda_2^{-2s}\,
a_{\lambda_1} a_{\lambda_2}\,b_{\lambda_3} 
\lesssim
\sum_{1 \ll \lambda_3 \ll \lambda_1 \sim \lambda_2} \lambda_3^{-\delta}
\lambda_2^{-s+\delta}\,
a_{\lambda_1} a_{\lambda_2}\,b_{\lambda_3} 
\]
for any small $\delta>0$, and this is summable whenever $s>0$. 
\end{itemize}

{$n\ge2$}. Using the improved factor \eqref{eq:bilinear_ball_improved} 
\[
G(\lambda)
\sim
\lambda_2^{\frac n2 - 2s - 1 + \delta}\,
 \lambda_3 ^{2-\omega}\,
\langle \lambda_3 \rangle^{s+\omega-2}.
\]

\begin{itemize}
\item If $\lambda_3 \lesssim 1 $, by the Cauchy--Schwarz inequality
\[
\begin{split}
R & \;\lesssim\;
\sum_{\lambda_3 \lesssim 1 \ll \lambda_1 \sim \lambda_2}
\lambda_3^{2-\omega}\,
\lambda_2^{\frac n2 - 2s- 1 + \delta}\,
a_{\lambda_1}\, a_{\lambda_2}\, b_{\lambda_3} \\
& \lesssim
\| (a_{\lambda_1}) \|_{\ell^2_{\lambda_1}}\,
\| (a_{\lambda_2}) \|_{\ell^2_{\lambda_2}}\,
\| (b_{\lambda_3}) \|_{\ell^2_{\lambda_3}}
\end{split}
\]
provided
\[
\omega<2,\qquad s>\frac{n-2}{4}+\frac{\delta}{2}.
\]

\item If $\lambda_3 \gg 1$,
\[
\begin{split}
R \; & \lesssim\;
\sum_{1 \ll \lambda_3 \ll \lambda_1 \sim \lambda_2}
\lambda_3^{s}\,
\lambda_2^{\frac n2 - 2s - 1 + \delta}\,
a_{\lambda_1}\, a_{\lambda_2}\, b_{\lambda_3} \\
& \sim
\sum_{1 \ll \lambda_3 \ll \lambda_1 \sim \lambda_2} \lambda_3^{-\delta}
\lambda_2^{\frac n2 - s - 1 + 2\delta}\,
a_{\lambda_1}\, a_{\lambda_2}\, b_{\lambda_3}  \\
& \lesssim
\| (a_{\lambda_1}) \|_{\ell^2_{\lambda_1}}\,
\| (a_{\lambda_2}) \|_{\ell^2_{\lambda_2}}\,
\| (b_{\lambda_3}) \|_{\ell^2_{\lambda_3}}
\end{split}
\]
provided $s> \frac{n }{2}-1+2\delta$.
\end{itemize}

\textbf{Case 3: $\lambda_2\gg 1$ and $\lambda_3\sim \lambda_2$}.  This is a low$\times$high$\to$high interaction.

$n=1$. 

\begin{itemize}
\item If $\lambda_1 \lesssim 1$, using \eqref{eq:bilinear-ball} get 
\[
R 
\;\lesssim\;
\sum_{\substack{\lambda_1 \lesssim 1 \\ 1 \ll \lambda_2 \sim \lambda_3}}
\lambda_1^{\frac12 + \omega}\,
a_{\lambda_1}\, a_{\lambda_2}\, b_{\lambda_3}
\]
which is summable, provided that $\omega > -\frac12.$

\item If $\lambda_1 \gg 1 $, 
\[
R
\;\lesssim\;
\sum
\lambda_1^{-s}\,
a_{\lambda_1}\, a_{\lambda_2}\, b_{\lambda_3} \lesssim \| (a_{\lambda_1}) \|_{\ell^2_{\lambda_1}}\,
\| (a_{\lambda_2}) \|_{\ell^2_{\lambda_2}}\,
\| (b_{\lambda_3}) \|_{\ell^2_{\lambda_3}}
\]
provided that $s > 0$.
\end{itemize}

$n\geq 2$.  In this case, we obtain $G(\lambda) \sim
\lambda_1^{\frac n2 - 1 + \delta + \omega}\,
\langle \lambda_1 \rangle^{-s-\omega}.$
\begin{itemize}
\item If $\lambda_1 \lesssim 1 $, 
\[
\begin{split}
R &
\;\lesssim\;
\sum_{\lambda_1 \lesssim 1 \ll \lambda_2 \sim \lambda_3}
\lambda_1^{\frac n2 - 1 + \delta + \omega}\,
a_{\lambda_1}\, a_{\lambda_2}\, b_{\lambda_3} \\
& \lesssim
\| (a_{\lambda_1}) \|_{\ell^2_{\lambda_1}}\,
\| (a_{\lambda_2}) \|_{\ell^2_{\lambda_2}}\,
\| (b_{\lambda_3}) \|_{\ell^2_{\lambda_3}}
\end{split} 
\]
requires $\omega>1-\frac n2-\delta.$

\item If $\lambda_1 \gg 1$,
\[
\begin{split}
R
& \;\lesssim\;
\sum_{1 \ll \lambda_1 \ll \lambda_2 \sim \lambda_3}
\lambda_1^{\frac n2 - 1 + \delta - s}\,
a_{\lambda_1}\, a_{\lambda_2}\, b_{\lambda_3} \\
& \lesssim
\| (a_{\lambda_1}) \|_{\ell^2_{\lambda_1}}\,
\| (a_{\lambda_2}) \|_{\ell^2_{\lambda_2}}\,
\| (b_{\lambda_3}) \|_{\ell^2_{\lambda_3}}
\end{split}
\]
provided $s>\frac n2-1+\delta.$
\end{itemize} \end{proof}

\begin{lemma}\label{lem:trilinear-estimate}
Let $n \geq 1$, $\frac{1}{2}< b < 1$, and $0 < T < 1$. Assume that
\[
0<\omega<2,\qquad s>0, \qquad \text{if } n=1,
\]
and
\[
\max \left(-\frac{n}{3},1-\frac n2\right)<\omega<2,\qquad s>\frac{n-1}{2}, \qquad \text{if } n\ge 2.
\]
Then
\begin{equation}\label{eq:trilinear-estimate}
\left\| |D|^2\langle D \rangle^{-2}(u_1 u_2 u_3) \right\|_{X^{s,b-1}_\omega(0,T)}
\lesssim
T^{1-b}
\prod_{j=1}^3 \|u_j\|_{X^{s,b}_\omega(0,T)}.
\end{equation}
\end{lemma}

\begin{lemma}\label{lem:loc-trilinear-estimate}
Let $n \geq 1$, $\frac{1}{2}< b < 1$, $0 < T < 1$, and let $\lambda_1,\lambda_2,\lambda_3,\lambda_4$ be dyadic numbers. Let
\[
\lambda_{\min}:=\min(\lambda_1,\lambda_2,\lambda_3), \qquad
\lambda_{\mathrm{med}}:=\operatorname{med}(\lambda_1,\lambda_2,\lambda_3), \qquad
\lambda_{\max}:=\max(\lambda_1,\lambda_2,\lambda_3).
\]
Then,
\begin{equation}\label{eq:loc-trilinear}
\left\| |D|^2\langle D \rangle^{-2} P_{\lambda_4}
\left( P_{\lambda_1} u_1 P_{\lambda_2} u_2 P_{\lambda_3} u_3 \right) \right\|_{L^2_T L^2_x}
\lesssim
\lambda_4^2 \langle \lambda_4 \rangle^{-2} B(\lambda)
\prod_{j=1}^3 \|P_{\lambda_j}u_j\|_{X^{0,b}_0},
\end{equation}
where
\begin{equation}\label{eq:tri-B-rough}
B(\lambda)\sim (\lambda_{\min}\lambda_{\mathrm{med}})^{\frac n2}.
\end{equation}
Moreover, if $\lambda_{\max}\gg1$, then one may improve $B$ to
\begin{equation}\label{eq:tri-B-improved}
B(\lambda)\lesssim
\begin{cases}
\lambda_{\min}^{\frac{1}{2}}, & \text{if } n=1, \\[4pt]
\lambda_{\min}^{\frac n2}\lambda_{\mathrm{med}}^{\frac n2-1+\delta}, & \text{if } n\ge2,
\end{cases}
\end{equation}
for sufficiently small $\delta>0$.
In addition, if $n=1$ and $\lambda_{\mathrm{med}}\gg1$, then one may further take
\begin{equation}\label{eq:tri-B-improved-1d-strong}
B(\lambda)\lesssim 1.
\end{equation}
\end{lemma}

\begin{proof}
By symmetry, we may assume $\lambda_1\le \lambda_2\le \lambda_3.$ Set $F:=P_{\lambda_1}u_1\,P_{\lambda_2}u_2\,P_{\lambda_3}u_3.$
By the Bernstein inequality \eqref{lem:bernstein-inequality}, and the Plancherel  equality
\begin{equation*}
\bigl\|\,|D|^2\langle D\rangle^{-2}P_{\lambda_4}F\bigr\|_{L_T^2L_x^2}
\lesssim
\lambda_4^2\langle \lambda_4\rangle^{-2}\|F\|_{L_T^2L_x^2}.
\end{equation*}
Thus it suffices to prove
\begin{equation}\label{eq:tri-core}
\|P_{\lambda_1}u_1\,P_{\lambda_2}u_2\,P_{\lambda_3}u_3\|_{L_T^2L_x^2}
\lesssim
B(\lambda)\prod_{j=1}^3 \|P_{\lambda_j}u_j\|_{X^{0,b}_0}.
\end{equation}

By the H\"older  inequality
\[
\|P_{\lambda_1}u_1\,P_{\lambda_2}u_2\,P_{\lambda_3}u_3\|_{L_T^2L_x^2}
\le
T^{\frac{1}{2}}
\|P_{\lambda_1}u_1\|_{L_T^\infty L_x^\infty}
\|P_{\lambda_2}u_2\|_{L_T^\infty L_x^\infty}
\|P_{\lambda_3}u_3\|_{L_T^\infty L_x^2}.
\]
Applying the Bernstein inequality to the first two factors and Sobolev embedding in time to all three factors, we obtain
\[
\|P_{\lambda_j}u_j\|_{L_T^\infty L_x^\infty}
\lesssim \lambda_j^{\frac n2}\|P_{\lambda_j}u_j\|_{L_T^\infty L_x^2}
\lesssim \lambda_j^{\frac n2}\|P_{\lambda_j}u_j\|_{X^{0,b}_0},
\qquad j=1,2,
\]
Since $T<1$, it follows that
\[
\|P_{\lambda_1}u_1\,P_{\lambda_2}u_2\,P_{\lambda_3}u_3\|_{L_T^2L_x^2}
\lesssim
(\lambda_1\lambda_2)^{\frac n2}
\prod_{j=1}^3 \|P_{\lambda_j}u_j\|_{X^{0,b}_0},
\]
which proves \eqref{eq:tri-core} with \eqref{eq:tri-B-rough}.

We now prove the improved bounds under the additional assumption $\lambda_3\gg1$.

\noindent \textbf{Case $n=1$.}
Choose the admissible pair $(q,r)=(4,\infty)$. By Lemma \ref{lem:localized-strichartz},
\[
\|P_{\lambda_3}u_3\|_{L_T^4L_x^\infty}
\lesssim
\|P_{\lambda_3}u_3\|_{X^{0,b}_0}.
\]
Using the H\"older  inequality,
\[
\begin{split}
\|P_{\lambda_1} u_1
\, P_{\lambda_2} u_2
\, P_{\lambda_3} u_3
\|_{L_T^2 L_x^2}
& \lesssim
\|P_{\lambda_1}u_1\|_{L_T^4L_x^\infty}
\|P_{\lambda_2}u_2\|_{L_T^\infty L_x^2}
\|P_{\lambda_3}u_3\|_{L_T^4L_x^\infty}.
\end{split}
\]
Applying the Bernstein inequality to the first factor and Sobolev embedding to the remaining factors, we obtain
\[
\|P_{\lambda_1}u_1\|_{L_T^4 L_x^\infty}
\lesssim
T^{\frac{1}{4}}\|P_{\lambda_1}u_1\|_{L_T^\infty L_x^\infty}
\lesssim
\lambda_1^{\frac{1}{2}}\|P_{\lambda_1}u_1\|_{L_T^\infty L_x^2}
\lesssim
\lambda_1^{\frac{1}{2}}\|P_{\lambda_1}u_1\|_{X^{0,b}_0},
\]
Hence
\[
\|P_{\lambda_1}u_1\,P_{\lambda_2}u_2\,P_{\lambda_3}u_3\|_{L_T^2L_x^2}
\lesssim
\lambda_1^{\frac{1}{2}}
\prod_{j=1}^3 \|P_{\lambda_j}u_j\|_{X^{0,b}_0},
\]
which proves \eqref{eq:tri-B-improved} for $n=1$.

If, in addition, $\lambda_2\gg1$, then we may use Lemma \ref{lem:localized-strichartz} on $\lambda_2$ also. Indeed,
\[
\|P_{\lambda_1}u_1\,P_{\lambda_2}u_2\,P_{\lambda_3}u_3\|_{L_T^2L_x^2}
\le
\|P_{\lambda_1}u_1\|_{L_T^\infty L_x^2}
\|P_{\lambda_2}u_2\|_{L_T^4L_x^\infty}
\|P_{\lambda_3}u_3\|_{L_T^4L_x^\infty}
\lesssim
\prod_{j=1}^3 \|P_{\lambda_j}u_j\|_{X^{0,b}_0},
\]
which proves \eqref{eq:tri-B-improved-1d-strong}.

\noindent \textbf{Case $n\ge2$.}
Fix $\delta\in(0,\frac14]$, and define
\[
r:=\frac{2n}{n-2+2\delta},
\qquad
q:=\frac{2}{1-\delta}.
\]
Then $(q,r)$ is Strichartz admissible pair. Observe that $\frac{1}{2}
= \frac{1}{\infty}+\frac{2}{qn}+\frac{1}{r}.$
Therefore the H\"older  inequality, Lemma \ref{lem:localized-strichartz}, the Bernstein inequality, and Sobolev embedding in time yield 
\[
\begin{split}
\|P_{\lambda_1}u_1\,P_{\lambda_2}u_2\,P_{\lambda_3}u_3\|_{L_T^2L_x^2} &
\le
T^{\frac12-\frac1q}
\|P_{\lambda_1}u_1\|_{L_T^\infty L_x^\infty}
\|P_{\lambda_2}u_2\|_{L_T^\infty L_x^{\frac{qn}{2}}}
\|P_{\lambda_3}u_3\|_{L_T^qL_x^r} \\
& \lesssim \lambda_1^{\frac n2} \lambda_2^{\,n(\frac12-\frac{2}{qn})} \|P_{\lambda_1}u_1\|_{X^{0,b}_0} \|P_{\lambda_2}u_2\|_{L_T^\infty L_x^2} \|P_{\lambda_3}u_3\|_{X^{0,b}_0}\\
& \lesssim \lambda_1^{\frac n2}\lambda_2^{\frac n2-1+\delta}
\prod_{j=1}^3 \|P_{\lambda_j}u_j\|_{X^{0,b}_0},
\end{split}
\]
which proves \eqref{eq:tri-B-improved} when $n\ge2$.
\end{proof}

To prove Lemma \ref{lem:trilinear-estimate}, we again use the time-restriction estimate \eqref{tao-eq}. So, it is enough to prove
\begin{equation}\label{eq:tri-step-1}
\bigl\||D|^2\langle D\rangle^{-2}(u_1u_2u_3)\bigr\|_{L_T^2H_\omega^s}
\lesssim
\prod_{j=1}^3 \|u_j\|_{X_\omega^{s,b}([0,T])}.
\end{equation}

\medskip
\begin{proof}[Proof of Lemma \ref{lem:trilinear-estimate}]

  By duality, it suffices to show that
\begin{equation}\label{eq:tri-duality}
\Bigg|
\int_0^T\!\!\int_{\mathbb{R}^n}
|D|^{2-\omega}\langle D\rangle^{s+\omega-2}
\Bigl(\prod_{j=1}^3 |D|^{\omega}\langle D\rangle^{-s-\omega}u_j\Bigr)
u_4\,dx\,dt
\Bigg| \lesssim
\prod_{j=1}^3 \|u_j\|_{X^{0,b}_0}\,
\|u_4\|_{L_T^2L_x^2}.
\end{equation}
Apply the Littlewood--Paley decompositions
$u_j=\sum_{\lambda_j}P_{\lambda_j}u_j, \, j=1,2,3,4,$ and set
\[
a_{\lambda_j}:=\|P_{\lambda_j}u_j\|_{X^{0,b}_0}, \qquad j=1,2,3,
\qquad
b_{\lambda_4}:=\|P_{\lambda_4}u_4\|_{L_T^2L_x^2}.
\]
Using the Cauchy--Schwarz inequality, Lemma \ref{lem:loc-trilinear-estimate}, and the Bernstein inequality, it is enough to prove that
\begin{equation}\label{eq:tri-sum} 
\underbrace{ \sum_{\lambda_1,\lambda_2,\lambda_3,\lambda_4} G(\lambda)\, a_{\lambda_1}a_{\lambda_2}a_{\lambda_3}b_{\lambda_4}}_{:=R} 
\lesssim
\|(a_{\lambda_1})\|_{\ell^2_{\lambda_1}}
\|(a_{\lambda_2})\|_{\ell^2_{\lambda_2}}
\|(a_{\lambda_3})\|_{\ell^2_{\lambda_3}}
\|(b_{\lambda_4})\|_{\ell^2_{\lambda_4}},
\end{equation}
where
\[
G(\lambda)
:=
B(\lambda)\,
\lambda_4^{2-\omega}\langle\lambda_4\rangle^{s+\omega-2}
\prod_{j=1}^3 \lambda_j^\omega\langle\lambda_j\rangle^{-s-\omega}.
\]
 We may assume by symmetry that $\lambda_1\le \lambda_2\le \lambda_3.$ Moreover, $\lambda_4\lesssim \lambda_3.$

\textbf{Case 1: $\lambda_3\lesssim 1$}. Using the rough bound \eqref{eq:tri-B-rough}, we obtain
\[
G(\lambda)\lesssim
\lambda_4^{2-\omega}(\lambda_1\lambda_2)^{\frac n2} \prod_{j=1}^3 \lambda_j^\omega \lesssim \lambda_4^{2-\omega}(\lambda_1\lambda_2 \lambda_3)^{\frac n3} \prod_{j=1}^3 \lambda_j^\omega .
\]
Applying the Cauchy--Schwarz inequality, we obtain
\[
R\lesssim
\sum_{\lambda_1,\lambda_2,\lambda_3,\lambda_4\lesssim 1}
\lambda_4^{2-\omega}
\lambda_1^{\omega+\frac n3}
\lambda_2^{\omega+\frac n3}
\lambda_3^{\omega+\frac{n}{3}}
\prod_{j=1}^3 a_{\lambda_j}\, b_{\lambda_4}
\lesssim
\prod_{j=1}^3 \|(a_{\lambda_j})\|_{\ell^2_{\lambda_j}}
\|(b_{\lambda_4})\|_{\ell^2_{\lambda_4}},
\]
provided $-\frac{n}{3}<\omega<2.$

\textbf{Case 2: $\lambda_1\sim\lambda_2\sim\lambda_3\gg 1$}. 
\paragraph{$n=1$.}
Since $\lambda_{\mathrm{med}}\gg 1$, Lemma \ref{lem:loc-trilinear-estimate} gives
\[
G(\lambda)
\lesssim
\lambda_4^{2-\omega}\langle\lambda_4\rangle^{s+\omega-2}\lambda_3^{-3s}.
\]
\begin{itemize}
    \item If $\lambda_4\lesssim 1$, 
\[
G(\lambda)\lesssim \lambda_4^{2-\omega}\lambda_3^{-3s},
\]
and $R$ is summable provided that $\omega<2$ and $s \ge 0$.
\item If $\lambda_4\gg 1$, then for any sufficiently small
$\delta>0$,
\[
G(\lambda)
\lesssim
\lambda_4^{-\delta}\lambda_3^{-2s+\delta},
\]
and $R$ is again summable as soon as $s>0$.
\end{itemize}

\paragraph{$n\ge 2$.} Using the improved bound \eqref{eq:tri-B-improved}, we obtain
\[
G(\lambda)
\lesssim
\lambda_3^{n-1+\delta-3s}
\lambda_4^{2-\omega}\langle\lambda_4\rangle^{s+\omega-2}.
\]
 
\begin{itemize}
    \item   If $\lambda_4\lesssim 1$, then
    \[
R\lesssim
\sum_{\substack{\lambda_4\lesssim 1\\ \lambda_1\sim\lambda_2\sim\lambda_3}}
\lambda_4^{2-\omega}
\lambda_3^{\,n-1+\delta-3s}
\prod_{j=1}^3 a_{\lambda_j}\, b_{\lambda_4}
\lesssim
\prod_{j=1}^3 \|(a_{\lambda_j})\|_{\ell^2_{\lambda_j}}
\|(b_{\lambda_4})\|_{\ell^2_{\lambda_4}},
\]
 provided
\[
\omega<2,
\qquad
s>\frac{n-1}3+\frac\delta3.
\]
\item If $\lambda_4\gg 1$, then using $\lambda_4\lesssim \lambda_3$ we get
\[
R\lesssim
\sum_{\substack{\lambda_4\gg 1\\ \lambda_1\sim\lambda_2\sim\lambda_3\gg 1}}
\lambda_4^{s}
\lambda_3^{\,n-1+\delta-3s}
\prod_{j=1}^3 a_{\lambda_j}\, b_{\lambda_4} \; \lesssim \sum_{\substack{\lambda_4\gg 1\\ \lambda_1\sim\lambda_2\sim\lambda_3\gg 1}}
\lambda_4^{-\delta}\lambda_3^{n-1-2s+2\delta}
\prod_{j=1}^3 a_{\lambda_j}\, b_{\lambda_4},
\]
which is summable provided
\[
s>\frac{n-1}2+\delta.
\]
\end{itemize}

\textbf{Case 3: $\lambda_1\lesssim\lambda_2\ll\lambda_3$} which implies $\lambda_4\sim \lambda_3$.

\paragraph{$n=1$.} We divide into three subcases.

\begin{itemize}
    \item $\lambda_2\lesssim 1$. Then also $\lambda_1\lesssim 1$,  and $G(\lambda)
\lesssim
\lambda_1^{\frac14+\omega}\lambda_2^{\frac14+\omega}.$ 
Hence $R$ is summable whenever $\omega>-\tfrac14$.
    \item  $\lambda_1\lesssim 1\ll\lambda_2$. Since $\lambda_{\mathrm{med}}\gg 1$, we may use $B(\lambda)\lesssim 1$. Therefore $G(\lambda)
\lesssim
\lambda_1^{\omega}\lambda_2^{-s}.$ $R$ is summable if $\omega>0$ and $s>0$.
\item  $1\ll\lambda_1\le \lambda_2$.
Again $B(\lambda)\lesssim 1$, so $G(\lambda)
\lesssim
\lambda_1^{-s}\lambda_2^{-s},$ $R$ is summable if $s>0$.
\end{itemize}

\paragraph{$n\ge 2$.} Using the improved bound \eqref{eq:tri-B-improved}, 
\[
G(\lambda)\sim
\lambda_1^{\frac n2+\omega}\langle \lambda_1\rangle^{-s-\omega}
\lambda_2^{\frac n2-1+\delta+\omega}\langle \lambda_2\rangle^{-s-\omega}.
\]

\begin{itemize}
    \item $\lambda_2\lesssim 1$. Then, $R$ is summable provided $\omega>\frac{1-n-\delta}{2}.$
\item  $\lambda_1\lesssim 1\ll\lambda_2$. In this case $G(\lambda)
\lesssim
\lambda_1^{\frac n2+\omega}
\lambda_2^{\frac n2-1+\delta-s}, \text{ and } R$ is summable provided 
\[
\omega>-\frac n2,
\qquad
s>\frac n2-1+\delta.
\]
\item  $1\ll\lambda_1\le \lambda_2$. Then,
\[
R\lesssim
\sum_{\substack{1\ll \lambda_1\le\lambda_2\ll\lambda_3\\ \lambda_4\sim\lambda_3}}
\lambda_1^{\frac n2-s}
\lambda_2^{\frac n2-1+\delta-s}
\prod_{j=1}^3 a_{\lambda_j}\,b_{\lambda_4}
\lesssim
\prod_{j=1}^3 \|(a_{\lambda_j})\|_{\ell^2_{\lambda_j}}
\|(b_{\lambda_4})\|_{\ell^2_{\lambda_4}},
\]
provided that $s>\frac{n}{2}$.
If $s\le \frac{n}{2}$, we have 
\[
R\lesssim
\sum_{\substack{1\ll \lambda_1\le\lambda_2\ll\lambda_3\\ \lambda_4\sim\lambda_3}}
\lambda_1^{-\delta}
\lambda_2^{-2s+n -1+2\delta}
\prod_{j=1}^3 a_{\lambda_j}\,b_{\lambda_4}
\lesssim
\prod_{j=1}^3 \|(a_{\lambda_j})\|_{\ell^2_{\lambda_j}}
\|(b_{\lambda_4})\|_{\ell^2_{\lambda_4}},
\]
provided that $s>\frac{n-1}{2}+\delta.$
\end{itemize}

\textbf{Case 4: $\lambda_1\ll\lambda_2\sim\lambda_3$}. We distinguish two possibilities $\lambda_4\sim\lambda_3$ and $\lambda_4\ll\lambda_3$.

\paragraph{$n=1$ and $\lambda_4\sim\lambda_3$.} We use \eqref{eq:tri-B-improved-1d-strong} to get $G(\lambda)
\lesssim
\lambda_1^{\omega}\langle\lambda_1\rangle^{-s-\omega}\lambda_3^{-s}.$
\begin{itemize}
    \item If $\lambda_1\lesssim 1$, then $R$ is summable provided $\omega>0$ and $s>0$.
    \item If $\lambda_1\gg 1$, then $G(\lambda)\lesssim \lambda_1^{-s}\lambda_3^{-s},$ and $R$ is summable for $s>0$.
\end{itemize}

\paragraph{$n=1$ and $\lambda_4\ll\lambda_3$.} Again $B(\lambda)\lesssim 1$. 
\begin{itemize}
    \item If $\lambda_4\lesssim 1$, then $G(\lambda)
\lesssim
\lambda_1^{\omega}\langle\lambda_1\rangle^{-s-\omega}\lambda_4^{2-\omega}\lambda_3^{-2s}$ and $R$ is summable provided $0<\omega<2$ and $s>0$.
\item If $\lambda_4\gg 1$, then using $\lambda_4\ll\lambda_3$, for any sufficiently small $\delta>0$, we have $G(\lambda)
\lesssim
\lambda_1^{\omega}\langle\lambda_1\rangle^{-s-\omega}\lambda_4^{-\delta}\lambda_3^{-s+\delta},$ so $R$ is summable under the same assumptions.
\end{itemize}

\paragraph{$n\ge 2$ and $\lambda_4\sim\lambda_3$.} Using the improved bound \eqref{eq:tri-B-improved}, $G(\lambda)
\lesssim
\lambda_1^{\frac n2+\omega}\langle\lambda_1\rangle^{-s-\omega}
\lambda_3^{\frac n2-1+\delta-s}.$
\begin{itemize}
    \item If $\lambda_1\lesssim 1$, $R$ is summable provided $\omega>-\frac n2, s>\frac n2-1+\delta.$
\item If $\lambda_1\gg 1$, then
\[
R\lesssim
\sum_{1\ll \lambda_1\ll \lambda_3\sim \lambda_4}
\lambda_1^{\frac n2-s}\lambda_3^{\frac n2-1+\delta-s}
\prod_{j=1}^3 a_{\lambda_j}\, b_{\lambda_4}.
\]
summable for $s>\frac n2$. If $s\le \frac n2$, then using $\lambda_1\ll \lambda_3$ we get $G(\lambda)\lesssim
\lambda_1^{-\delta}\lambda_3^{\,n-1+2\delta-2s}.$ $R$ is summable provided $s>\frac n2-\frac12+\delta.$
\end{itemize}

\paragraph{$n\ge 2$ and $\lambda_4\ll\lambda_3$.} In this case
\[
G(\lambda)
\lesssim
\lambda_1^{\frac n2+\omega}\langle\lambda_1\rangle^{-s-\omega}
\lambda_3^{\frac n2-1+\delta-2s}
\lambda_4^{2-\omega}\langle\lambda_4\rangle^{s+\omega-2}.
\]
We distinguish four subcases.
\begin{itemize}
    \item $\lambda_4\lesssim 1$ and $\lambda_1\lesssim 1$. Then $R$ is summable provided $-\frac n2< \omega<2,  s>\frac n4-\frac12+\frac12\delta.$
\item $\lambda_4\lesssim 1$ and $\lambda_1\gg 1$. Then $G(\lambda)\lesssim
\lambda_1^{\frac n2-s}
\lambda_4^{2-\omega}
\lambda_3^{\frac n2-1+\delta-2s}.$ If $s>\frac n2$, $R$ is summable provided $\omega<2$. If $s\le \frac n2$, then $G(\lambda)\lesssim
\lambda_1^{-\delta}
\lambda_4^{2-\omega}
\lambda_3^{\,n-1+2\delta-3s}.$ Hence $R$ is summable provided $\omega<2, s>\frac n3-\frac13+\frac23\delta.$
\item  $\lambda_4\gg 1$ and $\lambda_1\lesssim 1$. Since $\lambda_4\ll\lambda_3$, we use $\lambda_4^s\lesssim \lambda_4^{-\delta}\lambda_3^{s+\delta},$
and obtain summability conditions $\omega>-\frac n2, s>\frac n2-1+2\delta.$
\item $\lambda_4\gg 1$ and $\lambda_1\gg 1$. Using again the relation $\lambda_4\ll \lambda_3$, we obtain $G(\lambda)\lesssim
\lambda_1^{\frac n2-s}
\lambda_4^{-\delta}
\lambda_3^{\frac n2-1+2\delta-s}.$ If $s>\frac n2$, then the summation is uniformly bounded. If $s\le \frac n2$, then \[
G(\lambda)\lesssim
\lambda_1^{-\delta}
\lambda_4^{-\delta}
\lambda_3^{\,n-1+3\delta-2s}.
\]
Hence, $R$ is summable provided $s>\frac n2-\frac12+\frac32\delta.$
\end{itemize}
Collecting all cases, we conclude that \eqref{eq:tri-sum} holds. Therefore the dual estimate is proved, which yields \eqref{eq:tri-step-1}. \end{proof}

Applying Lemma \ref{lem:bilinear-estimate} to  the Duhamel term gives
\[
\left\|
|D|^2\langle D\rangle^{-2}u^2
\right\|_{X_\omega^{s,b-1}([0,T])}
\lesssim
T^{1-b}\|u\|_{X_\omega^{s,b}([0,T])}^2.
\]
Therefore
\[
\|\Phi(u)\|_{X_\omega^{s,b}([0,T])}
\lesssim
\|g\|_{H_\omega^s}
+
\|h\|_{H_\omega^{s-2}}
+
T^{1-b}\|u\|_{X_\omega^{s,b}([0,T])}^2.
\]
Similarly, for two functions \(u,v\in X_\omega^{s,b}([0,T])\), $u^2-v^2=(u-v)(u+v),$ Lemma \ref{lem:bilinear-estimate} implies
\[
\|\Phi(u)-\Phi(v)\|_{X_\omega^{s,b}([0,T])}
\lesssim
T^{1-b}
\left(
\|u\|_{X_\omega^{s,b}([0,T])}
+
\|v\|_{X_\omega^{s,b}([0,T])}
\right)
\|u-v\|_{X_\omega^{s,b}([0,T])}.
\]
Let $M=\|g\|_{H_\omega^s}+\|h\|_{H_\omega^{s-2}}$. Choose \(R=2CM\), where \(C>0\) is the constant. Then choose \(T>0\) sufficiently small so that
\[
CT^{1-b}R\le \frac12.
\]
With this choice, \(\Phi\) maps the ball
\[
B_R=\{u\in X_\omega^{s,b}([0,T]):\|u\|_{X_\omega^{s,b}([0,T])}\le R\}
\]
into itself and is a contraction on \(B_R\). Hence, by the Banach fixed-point theorem, there exists a unique solution
\[
u\in X_\omega^{s,b}([0,T]).
\]
The same estimates also imply continuous dependence on the initial data. This proves the local well-posedness result for quadratic nonlinearity. Using Lemma \ref{lem:trilinear-estimate} on the cubic case, we obtain the same result. 

\subsection*{Remarks}\label{section4}
For the future work, we can consider more general combined-power nonlinearities. 
In addition, the estimates also suggest some directions for   future work. One natural direction is the persistence of spatial
analyticity. 
Another direction is global well-posedness. For generalized and higher-order
Boussinesq equations, global existence has been obtained under various assumptions
on the sign of the nonlinearity, the size of the data, or the energy level
\cite{BonaSachs1988,Linares1993}. However, in our case, global well-posedness requires
a detailed understanding of the modified energy conservation law. Hence, it will be interesting to study the global well-posedness.

\bibliographystyle{plain}
\bibliography{references}

\end{document}